\documentclass[11pt,a4paper]{article}
\usepackage{mathrsfs}
\usepackage{indentfirst}
\usepackage{amsmath,amssymb,amsfonts,amsthm,graphics}
\usepackage{makeidx}
\usepackage{color}

\newtheorem{theorem}{Theorem}

\newtheorem{proposition}{Proposition}
\newtheorem{lemma}{Lemma}

\newtheorem{corollary}{Corollary}

\newtheorem{observation}{Observation}

\begin{document}
\title{\Large\bf Nordhaus-Gaddum-type results for the\\
generalized edge-connectivity of graphs\footnote{Supported by NSFC
No.11071130 and the ``973'' project.}}
\author{\small Xueliang Li, Yaping Mao
\\
\small Center for Combinatorics and LPMC-TJKLC
\\
\small Nankai University, Tianjin 300071, China
\\
\small lxl@nankai.edu.cn; maoyaping@ymail.com}
\date{}
\maketitle
\begin{abstract}
Let $G$ be a graph, $S$ be a set of vertices of $G$, and
$\lambda(S)$ be the maximum number $\ell$ of pairwise edge-disjoint
trees $T_1, T_2, \cdots, T_{\ell}$ in $G$ such that $S\subseteq
V(T_i)$ for every $1\leq i\leq \ell$. The generalized
$k$-edge-connectivity $\lambda_k(G)$ of $G$ is defined as
$\lambda_k(G)= min\{\lambda(S) | S\subseteq V(G) \ and \ |S|=k\}$.
Thus $\lambda_2(G)=\lambda(G)$. In this paper, we consider the
Nordhaus-Gaddum-type results for the parameter $\lambda_k(G)$. We
determine sharp upper and lower bounds of
$\lambda_k(G)+\lambda_k(\overline{G})$ and $\lambda_k(G)\cdot
\lambda_k(\overline{G})$ for a graph $G$ of order $n$, as well as
for a graph of order $n$ and size $m$. Some graph classes attaining
these bounds are also given.

\end{abstract}

\section{Introduction}
All graphs considered in this paper are undirected, finite and
simple. We refer to the book \cite{bondy} for graph theoretical
notation and terminology not described here. For a graph $G(V,E)$
and a set $S\subseteq V$ of at least two vertices, \emph{an
$S$-Steiner tree} or \emph{an Steiner tree connecting $S$} (Shortly,
\emph{a Steiner tree}) is a subgraph $T(V',E')$ of $G$ which is a
tree such that $S\subseteq V'$. Two Steiner trees $T$ and $T'$
connecting $S$ are \emph{edge-disjoint} if $E(T)\cap
E(T')=\varnothing$. The \emph{Steiner Tree Packing Problem} for a
given graph $G(V,E)$ and $S\subseteq V(G)$ asks to find a set of
maximum number of edge-disjoint $S$-Steiner trees in $G$. This
problem has obtained wide attention and many results have been
worked out, see \cite{Kriesell1, Kriesell2, West}. The problem for
$S=V(G)$ is called the \emph{Spanning Tree Packing Problem}. For any
graph $G$ of order $n$, the \emph{spanning tree packing number} or
\emph{$STP$ number}, is the maximum number of edge-disjoint spanning
trees contained in $G$. For the $STP$ number, Palmer gave a good
survey, see \cite{Palmer}.

Recently, we introduced the concept of generalized edge-connectivity
of a graph $G$ in \cite{LMS}. For $S\subseteq V(G)$, the
\emph{generalized local edge-connectivity} $\lambda(S)$ is the
maximum number of edge-disjoint trees in $G$ connecting $S$. Then
the \emph{generalized $k$-edge-connectivity} $\lambda_k(G)$ of $G$
is defined as $\lambda_k(G)= min\{\lambda(S) : S\subseteq V(G) \ and
\ |S|=k\}$. Thus $\lambda_2(G)=\lambda(G)$. Set $\lambda_k(G)=0$
when $G$ is disconnected. We call it the generalized
$k$-edge-connectivity since Chartrand et al. in \cite{Chartrand1}
introduced the concept of generalized (vertex) connectivity in 1984.
There have been many results on the generalized connectivity, see
\cite{LLSun, LL, LLZ, LMS}.

One can see that the Steiner Tree Packing Problem studies local
properties of graphs, but the generalized edge-connectivity focuses
on global properties of graphs. Actually, the $STP$ number of a
graph $G$ is just $\lambda_n(G)$.

In addition to being natural combinatorial measures, the Steiner
Tree Packing Problem and the generalized edge-connectivity can be
motivated by their interesting interpretation in practice as well as
theoretical consideration. For the practical backgrounds, we refer
to \cite{Grotschel1, Grotschel2, Sherwani}.

From a theoretical perspective, both extremes of this problem are
fundamental theorems in combinatorics. One extreme of the problem is
when we have two terminals. In this case internally (edge-)disjoint
trees are just internally (edge-)disjoint paths between the two
terminals, and so the problem becomes the well-known Menger theorem.
The other extreme is when all the vertices are terminals. In this
case internally disjoint trees and edge-disjoint trees are just
spanning trees of the graph, and so the problem becomes the
classical Nash-Williams-Tutte theorem.

\begin{theorem}(Nash-Williams \cite{Nash}, Tutte \cite{Tutte})\label{th1}
A multigraph $G$ contains a system of $\ell$ edge-disjoint spanning
trees if and only if
$$
\|G/\mathscr{P}\|\geq \ell(|\mathscr{P}|-1)
$$
holds for every partition $\mathscr{P}$ of $V(G)$, where
$\|G/\mathscr{P}\|$ denotes the number of crossing edges in $G$,
i.e., edges between distinct parts of $\mathscr{P}$.
\end{theorem}

\begin{corollary}\label{cor1}
Every $2\ell$-edge-connected graph contains a system of $\ell$
edge-disjoint spanning trees.
\end{corollary}

Let $\mathcal {G}(n)$ denote the class of simple graphs of order $n$
and $\mathcal {G}(n,m)$ the subclass of $\mathcal {G}(n)$ having $m$
edges. Give a graph theoretic parameter $f(G)$ and a positive
integer $n$, the \emph{Nordhaus-Gaddum(\textbf{N-G}) Problem} is to
determine sharp bounds for: $(1)$ $f(G)+f(\overline{G})$ and $(2)$
$f(G)\cdot f(\overline{G})$, as $G$ ranges over the class $\mathcal
{G}(n)$, and characterize the extremal graphs. The Nordhaus-Gaddum
type relations have received wide investigations. Recently,
Aouchiche and Hansen published a survey paper on this subject, see
\cite{Aouchiche}.

In this paper, we study $\lambda_k(G)+\lambda_k(\overline{G})$ and
$\lambda_k(G)\cdot \lambda_k(\overline{G})$ for the parameter
$\lambda_k(G)$ where $G\in \mathcal {G}(n)$ and $G\in \mathcal
{G}(n,m)$.

\section{Nordhaus-Gaddum-type results in $\mathcal {G}(n)$}

The following observation is easily seen.

\begin{observation}\label{obs1}
$(1)$ If $G$ is a connected graph, then $1\leq \lambda_k(G)\leq
\lambda(G)\leq \delta(G)$;

$(2)$ If $H$ is a spanning subgraph of $G$, then $\lambda_k(H)\leq
\lambda_k(G)$.

$(3)$ Let $G$ be a connected graph with minimum degree $\delta$. If
$G$ has two adjacent vertices of degree $\delta$, then
$\lambda_k(G)\leq \delta-1$.
\end{observation}

Alavi and Mitchem in \cite{Alavi} considered Nordhaus-Gaddum-type
results for the connectivity and edge-connectivity parameters. In
\cite{LMS} we were concerned with analogous inequalities involving
the generalized $k$-connectivity and generalized
$k$-edge-connectivity. We showed that $1\leq
\lambda_k(G)+\lambda_k(\overline{G})\leq n-\lceil k/2 \rceil$, but
this is just a starting result and now we will further study the
Nordhaus-Guddum type relations.

To start with, let us recall the Harary graph $H_{n,d}$ on $n$
vertices, which is constructed by arranging the $n$ vertices in
circular order and spreading the $d$ edges around the boundary in a
nice way, keeping the chords as short as possible. They have the
maximum connectivity for their size and
$\kappa(H_{n,d})=\lambda(H_{n,d})= \delta(H_{n,d})=d$. Palmer
\cite{Palmer} gave the $STP$ number of some special graph classes.

\begin{lemma}\cite{Palmer}\label{lem1}
$(1)$ The $STP$ number of a complete bipartite graph $K_{a,b}$ is
$\lfloor\frac{ab}{a+b-1}\rfloor$.

$(2)$ The $STP$ number of a Harary graph $H_{n,d}$ is $\lfloor d
/2\rfloor$.
\end{lemma}

Corresponding to $(1)$ of Observation \ref{obs1}, we can obtain a
sharp lower bound for the generalized $k$-edge-connectivity by
Corollary \ref{cor1}. Actually, a connected graph $G$ contains
$\lfloor\frac{1}{2}\lambda(G)\rfloor$ spanning trees. Each of them
is also a Steiner tree connecting $S$. So the following proposition
is immediate.

\begin{proposition}\label{pro1}
For a connected graph $G$ of order $n$ and $3\leq k\leq n$,
$\lambda_k(G)\geq \lfloor\frac{1}{2}\lambda(G)\rfloor$. Moreover,
the lower bound is sharp.
\end{proposition}

In order to show the sharpness of this lower bound for $k=n$, we
consider the Harary graph $H_{n,2r}$. Clearly, $\lambda(G)=2r$. From
$(2)$ of Lemma \ref{lem1}, $H_{n,2r}$ contains $r$ spanning trees,
that is, $\lambda_n(H_{n,2r})=r$. So
$\lambda_n(H_{n,2r})=\lfloor\frac{1}{2}\lambda(G)\rfloor$. For
general $k \ (3\leq k\leq n)$, one can check that the cycle $C_n$
can attain the lower bound since
$\frac{1}{2}\lambda(C_n)=1=\lambda_k(C_n)$.

The following proposition indicates that the monotone properties of
$\lambda_k$, that is, $\lambda_n\leq \lambda_{n-1}\leq \cdots
\lambda_4\leq \lambda_3\leq \lambda$, is true for $2\leq k\leq n$.

\begin{proposition}\label{pro2}
For two integers $k$ and $n$ with $2\leq k\leq n-1$, and a connected
graph $G$, $\lambda_{k+1}(G)\leq \lambda_{k}(G)$.
\end{proposition}

\begin{proof}
Assume $3\leq k\leq n-1$. Set $\lambda_{k+1}(G)=\ell$. For each
$S\subseteq V(G)$ with $|S|=k$, we let $S'=S\cup \{u\}$, where
$u\notin S$. Since $\lambda_{k+1}(G)=\ell$, there exist $\ell$
edge-disjoint trees connecting $S'$. These trees are also $\ell$
edge-disjoint trees connecting $S$. So $\lambda_{k}(G)\geq \ell$ and
$\lambda_{k+1}(G)\leq \lambda_{k}(G)$. Combining this with $(1)$ of
Observation \ref{obs1}, we get that $\lambda_{k+1}(G)\leq
\lambda_{k}(G)$ for $2\leq k\leq n-1$.
\end{proof}

Now we give the lower bounds of
$\lambda_k(G)+\lambda_k(\overline{G})$ and $\lambda_k(G)\cdot
\lambda_k(\overline{G})$.

\begin{lemma}\label{lem2}
Let $G\in \mathcal {G}(n)$. Then

$(1)$ $\lambda_k(G)+\lambda_k(\overline{G})\geq 1$;

$(2)$ $\lambda_k(G)\cdot \lambda_k(\overline{G})\geq 0$.

Moreover, the two lower bounds are sharp.
\end{lemma}

\begin{proof}
$(1)$ If $\lambda_k(G)+\lambda_k(\overline{G})=0$, then
$\lambda_k(G)=\lambda_k(\overline{G})=0$, that is, $G$ and
$\overline{G}$ are all disconnected, which is impossible, and so
$\lambda_k(G)+\lambda_k(\overline{G})\geq 1$.

$(2)$ By definition, $\lambda_k(G)\geq 0$ and
$\lambda_k(\overline{G})\geq 0$, and so $\lambda_k(G)\cdot
\lambda_k(\overline{G})\geq 0$.
\end{proof}

The following observation indicates the graphs attaining the lower
bound of $(1)$ in Lemma \ref{lem2}.

\begin{observation}\label{obs2}
$\lambda_k(G)\cdot \lambda_k(\overline{G})=0$ if and only if $G$ or
$\overline{G}$ is disconnected.
\end{observation}

In \cite{LMS} we obtained the exact value of the generalized
$k$-edge-connectivity of a complete graph $K_n$.

\begin{lemma}\cite{LMS}\label{lem3}
For two integers $n$ and $k$ with $2\leq k\leq n$,
$\lambda_k(K_n)=n-\lceil k/2\rceil$.
\end{lemma}

For a connected graph $G$ of order $n$, we know that $1\leq
\lambda_k(G) \leq \lambda_k(K_n)=n-\lceil k/2\rceil$. In \cite{LMS}
we characterized the graphs attaining the upper bound.

\begin{lemma}\cite{LMS}\label{lem4}
For a connected graph $G$ of order $n$ with $3\leq k\leq n$,
$\lambda_k(G)=n-\lceil\frac{k}{2}\rceil$ if and only if $G=K_n$ for
$k$ even; $G=K_n\setminus M$ for $k$ odd, where $M$ is an edge set
such that $0\leq |M|\leq \frac{k-1}{2}$.
\end{lemma}

As we know, it is difficult to characterize the graphs with
$\lambda_k(G)=1$, even with $\lambda_3(G)=1$. So we want to add some
conditions to attack such a problem. Motivated by such an idea, we
hope to characterize the graphs with
$\lambda_k(G)+\lambda_k(\overline{G})=1$. Actually, the
Norhaus-Gaddum-type problems also need to characterize the extremal
graphs attaining the bounds.

Before studying the lower bounds of $\lambda_k(G)+
\lambda_k(\overline{G})$ and $\lambda_k(G)\cdot
\lambda_k(\overline{G})$, we give some graph classes (Every element
of each graph class has order $n$), which will be used later.

For $n\geq 5$, $\mathcal {G}_n^1$ is a graph class as shown in
Figure 1 $(a)$ such that $\lambda(G)=1$ and $d_{G}(v_1)=n-1$ for
$G\in \mathcal {G}_n^1$, where $v_1\in V(G)$; $\mathcal {G}_n^2$ is
a graph class as shown in Figure 1 $(b)$ such that $\lambda(G)=2$
and $d_{G}(u_1)=n-1$ for $G\in \mathcal {G}_n^2$, where $u_1\in
V(G)$; $\mathcal {G}_n^3$ is a graph class as shown in Figure 1
$(c)$ such that $\lambda(G)=2$ and $d_{G}(v_1)=n-1$ for $G\in
\mathcal {G}_n^3$, where $v_1\in V(G)$; $\mathcal {G}_n^4$ is a
graph class as shown in Figure 1 $(d)$ such that $\lambda(G)=2$.

\begin{figure}[h,t,b,p]
\begin{center}
\scalebox{0.8}[0.8]{\includegraphics{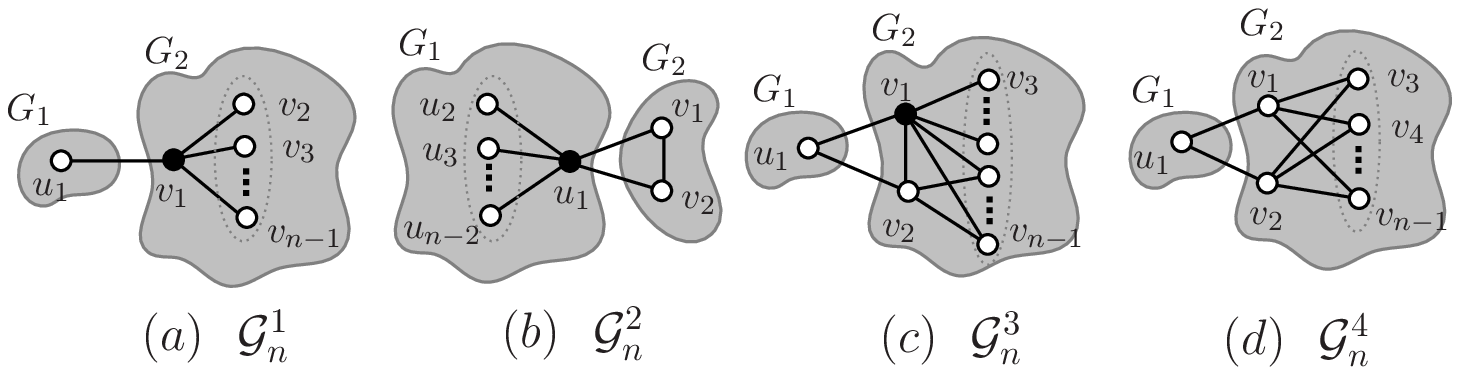}}\\
Figure 1. Graphs for Proposition \ref{pro3} (The degree of a black
vertex is $n-1$).
\end{center}
\end{figure}

The following observation and lemma are some preparations for
Proposition \ref{pro3}.

For $n\geq 5$, let $K_{2,n-2}^{+}$ and $K_{2,n-2}^{++}$ be two
graphs obtained from the complete bipartite graph $K_{2,n-2}$ by
adding one and two edges on the part having $n-2$ vertices,
respectively.

\begin{observation}\label{obs3}
$(1)$ $\lambda_n(K_{2,n-2}^{++})\geq 2$; $(2)$
$\lambda_{n-1}(K_{2,n-2}^{+})\geq 2$,
$\lambda_{n}(K_{2,n-2}^{+})=1$; $(3)$ $\lambda_{n-2}(K_{2,n-2})\geq
2$, $\lambda_{n}(K_{2,n-2})=\lambda_{n-1}(K_{2,n-2})=1$.
\end{observation}

\begin{proof}
$(1)$ As shown in Figure 2 $(a)$, $\lambda_n(K_{2,n-2}^{++})\geq 2$.

$(2)$ As shown in Figure 2 $(b)$, we have
$\lambda_{n-1}(K_{2,n-2}^{+})\geq 2$. Since
$|E(K_{2,n-2}^{+})|=2(n-2)+1$, $\lambda_{n}(K_{2,n-2}^{+})\leq
\lfloor\frac{2(n-2)+1}{n-1}\rfloor$, which implies that
$\lambda_{n}(K_{2,n-2}^{+})\leq 1$. Since $K_{2,n-2}^{+}$ is
connected, $\lambda_{n}(K_{2,n-2}^{+})=1$.

$(3)$ As shown in Figure 2 $(c)$, it follows that
$\lambda_{n-2}(K_{2,n-2})\geq 2$. Let $U=\{u_1,u_2\}$ and
$W=\{w_1,w_2, \cdots,w_{n-2}\}$ be two parts of the complete
bipartite graph $K_{2,n-2}$. Choose
$S=\{u_1,u_2,w_1,w_2,\cdots,w_{n-3}\}$. If there exists an $S$-tree
containing vertex $w_{n-2}$, then this tree will use $n-1$ edges of
$E(K_{2,n-2})$, which implies that $\lambda_{n-1}(K_{2,n-2})\leq 1$
since $|E(K_{2,n-2})|=2(n-2)$. Suppose that there is no $S$-tree
containing vertex $w_2$. Pick up a such tree, say $T$. Then there
exists a vertex of degree $2$ in $T$, which implies that there is no
other $S$-tree in $K_{2,n-2}$. So $\lambda_{n-1}(K_{2,n-2})\leq 1$.
Since $K_{2,n-2}$ is connected, $\lambda_{n-1}(K_{2,n-2})=1$. From
Proposition \ref{pro2}, $\lambda_{n}(K_{2,n-2})=1$.
\end{proof}

\begin{figure}[h,t,b,p]
\begin{center}
\scalebox{0.9}[0.9]{\includegraphics{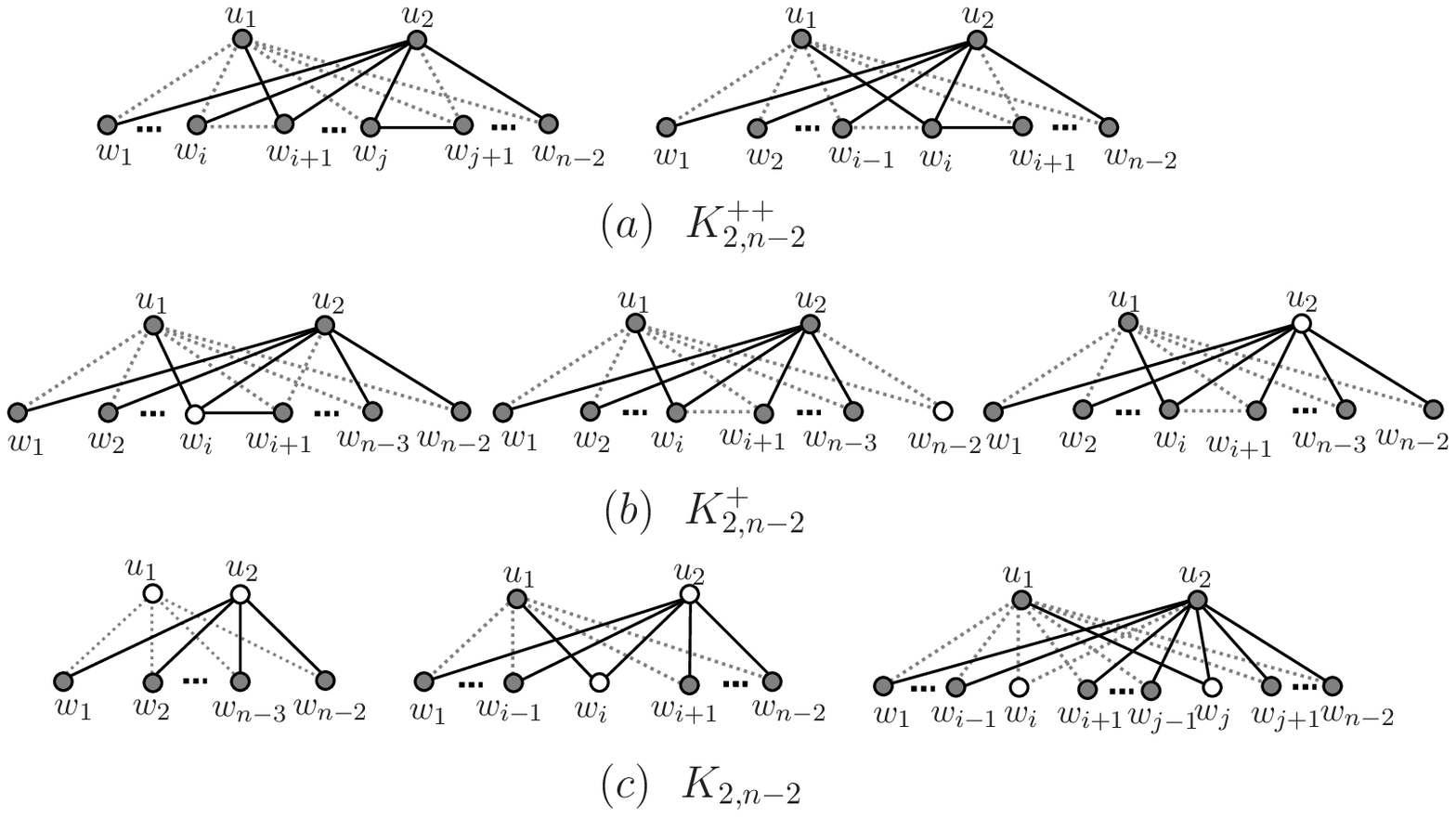}}\\
Figure 2. Graphs for Observation \ref{obs2}.
\end{center}
\end{figure}

\begin{lemma}\label{lem5}
Let $G$ be a connected graph. If $\lambda(G)=3$ and there exists a
vertex $u\in V(G)$ such that $d_G(u)=n-1$, then $\lambda_k(G)\geq 2$
for $3\leq k\leq n$.
\end{lemma}

\begin{proof}
Let $G_1,\cdots,G_r$ be the connected components of $G\setminus u$.
Since $\lambda(G)=3$, it follows that $\delta (G_i)\geq 2 \ (1\leq
i\leq r)$. Let $|V(G_i)|=n_i \ (1\leq i\leq r)$ and
$V(G_i)=\{v_{i1},v_{i2},\cdots, v_{in_i}\}$. Then there exists an
edge, without loss of generality, say $e_i=v_{i1}v_{i2}\in E(G_i)$
such that $G_i\setminus e_i$ is connected for $1\leq i\leq r$. Thus
$G_i\setminus e_i$ contains a spanning tree, say $T_i \ (1\leq i\leq
r)$. The trees $T=uv_{11}\cup T_1\cup uv_{21}\cup T_2\cup \cdots
\cup uv_{r1}\cup T_r$ and $T'=v_{11}v_{12}\cup uv_{12}\cup \cdots
\cup uv_{1n_1}\cup v_{21}v_{22}\cup uv_{22}\cup \cdots \cup
uv_{2n_2}\cup \cdots \cup v_{r1}v_{r2}\cup uv_{r2}\cup \cdots \cup
uv_{rn_r}$ are two spanning trees of $G$, that is, $\lambda_n(G)\geq
2$. Combining this with Proposition \ref{pro2}, $\lambda_k(G)\geq 2$
for $3\leq k\leq n$.
\end{proof}

\begin{proposition}\label{pro3}
$\lambda_k(G)+\lambda_k(\overline{G})=1$ if and only if $G$
(symmetrically, $\overline{G}$) satisfies one of the following
conditions:

$(1)$ $G\in \mathcal {G}_n^1$ or $G\in \mathcal {G}_n^2$;

$(2)$ $G\in \mathcal {G}_n^3$ and there exists a component $G_i$ of
$G\setminus v_1$ such that $G_i$ is a tree and $|V(G_i)|<k$;

$(3)$ $G\in \{K_{2,n-2}^{+},K_{2,n-2}\}$ for $k=n$ and $n\geq 5$, or
$G\in \{P_3,C_3\}$ for $k=n=3$, or $G\in \{C_4,K_4\setminus e\}$ for
$k=n=4$, or $G=K_{3,3}$ for $k=n=6$, or $G=K_{2,n-2}$ for $k=n-1$
and $n\geq 5$, or $G=C_4$ for $k=n-1=3$.
\end{proposition}

\begin{proof}

\emph{Necessity}. Let $G$ be a graph satisfying one of the
conditions of $(1)$, $(2)$ and $(3)$. One can see that $G$ is
connected and its complement $\overline{G}$ is disconnected. Thus
$\lambda_k(G)+\lambda_k(\overline{G})=\lambda_k(G)$ and
$\lambda_k(G)\geq 1$. We only need to show that $\lambda_k(G)\leq 1$
for each graph $G$ satisfying one of the conditions of $(1)$, $(2)$
and $(3)$. For $G\in \mathcal {G}_n^1$, since $\delta(G)=1$ we have
$\lambda_k(G)\leq 1$ by $(1)$ of Observation \ref{obs1}. For $G\in
\mathcal {G}_n^2$, it follows that $\lambda_k(G)\leq \delta(G)-1=1$
by $(3)$ of Observation \ref{obs1} since
$d_G(v_1)=d_G(v_2)=\delta(G)=2$. Suppose $G\in \mathcal {G}_n^3$ and
there exists a connected component $G_i$ of $G\setminus v_1$ such
that $G_i$ is a tree and $|V(G_i)|<k$. Set
$V(G_i)=\{v_{i1},v_{i2},\cdots,v_{in_i}\}$. We choose $S\subseteq
V(G)$ such that $V(G_i)\cup \{v_{1}\}=S'\subseteq S$. Then
$|E(G[S'])|=2n_i-1$. Since every spanning tree of $G[S']$ uses
$n_i-1$ edges of $E(G[S'])$, there exists at most one spanning tree
of $G[S']$, which implies that there is at most one tree connecting
$S$ in $G$. So $\lambda_k(G)\leq 1$. For $G=K_{2,n-2}^{+}$,
$\lambda_n(G)=1$ by $(2)$ of Observation \ref{obs3}. For
$G=K_{2,n-2}$, by $(3)$ of Observation \ref{obs3}, we have
$\lambda_{n}(K_{2,n-2})=\lambda_{n-1}(K_{2,n-2})=1$. For
$G=K_{3,3}$, $\lambda_n(G)\leq
\lfloor\frac{|E(G)|}{n-1}\rfloor=\lfloor\frac{9}{5}\rfloor=1$. For
$G\in \{P_3, C_3, C_4, K_4\setminus e\}$, one can check that
$\lambda_k(G)\leq 1$ for $k=n$ or $k=n-1$. From these together with
$\lambda_k(G)\geq 1$, we have
$\lambda_k(G)+\lambda_k(\overline{G})=\lambda_k(G)=1$.

\emph{Sufficiency}. Suppose
$\lambda_k(G)+\lambda_k(\overline{G})=1$. Then $\lambda_k(G)=1$ and
$\lambda_k(\overline{G})=0$, or $\lambda_k(\overline{G})=1$ and
$\lambda_k(G)=0$. By symmetry, without loss of generality, we let
$\lambda_k(G)=1$ and $\lambda_k(\overline{G})=0$. From these
together with Proposition \ref{pro1}, $\lambda(\overline{G})=0$ and
$1\leq \lambda(G)\leq 3$. So we have the following three cases to
consider.

\emph{Case 1.}~~$\lambda(G)=1$.

For $n=3$, one can check that $G=P_3$ satisfies $\lambda(G)=1$ but
$\lambda(\overline{G})=0$. Now we assume $n\geq 4$. Since
$\lambda(G)=1$, there exists at least one cut edge in $G$, say
$e=u_1v_1$. Let $G_1$ and $G_2$ be two connected components of
$G\setminus e$ such that $u_1\in V(G_1)$ and $v_1\in V(G_2)$. Set
$V(G_1)=\{u_1,u_2,\cdots,u_{n_1}\}$ and
$V(G_2)=\{v_1,v_2,\cdots,v_{n_2}\}$, where $n_1+n_2=n$. Suppose
$n_i\geq 2 \ (i=1,2)$. For any $u_i,u_j\in V(G_1)$, $u_i$ and $u_j$
are connected in $\overline{G}$ since there exists a path
$u_iv_2u_j$ in $\overline{G}$; for any $v_i,v_j\in V(G_2)$, $v_i$
and $v_j$ are connected in $\overline{G}$ since there exists a path
$v_iu_2v_j$ in $\overline{G}$; for any $u_i\in V(G_1)$ and $v_j\in
V(G_2)$ ($i\neq 1$ or $j\neq 1$), $v_iv_j\in E(\overline{G})$.
Clearly, the path $u_1v_2u_2v_1$ connects $u_1$ and $v_1$ in
$\overline{G}$. So $\overline{G}$ is connected, a contradiction.
Thus $n_1=1$ or $n_2=1$. Without loss of generality, let $n_1=1$.
Then $V(G_1)=\{u_1\}$ and $V(G_2)=\{v_1,v_2,\cdots,v_{n-1}\}$.
Clearly, $G$ is a graph obtained from $G_2$ by attaching the edge
$e=u_1v_1$. Since $u_1v_j\notin E(G) \ (1\leq j\leq n-1)$,
$u_1v_j\in E(\overline{G})$. If $d_{G}(v_1)\leq n-2$, then there
exists one vertex $v_j$ such that $v_1v_j\in E(\overline{G})$, which
results in $\lambda(\overline{G})\geq 1$, a contradiction. So
$d_{G}(v_1)=n-1$ and $G\in \mathcal {G}_n^1$ (See Figure 1 $(a)$).

\emph{Case 2.}~~$\lambda(G)=2$.

For $n=3,4$, the graph $G\in \{C_3,C_4,K_4\setminus e\}$ satisfies
that $\lambda(G)=2$ and $\lambda(\overline{G})=0$. Since
$\lambda_3(C_3)=1$, $\lambda_3(C_4)=1$, $\lambda_4(C_4)=1$,
$\lambda_3(K_4\setminus e)=2$ and $\lambda_4(K_4\setminus e)=1$, we
have $G=C_3$ for $k=n=3$; $G\in \{C_4,K_4\setminus e\}$ for $k=n=4$;
$G=C_4$ for $k=n-1=3$. Now we assume $n\geq 5$. Since
$\lambda(G)=2$, there exists an edge cut $M$ such that $|M|=2$. Let
$G_1$ and $G_2$ be two connected components of $G\setminus M$,
$V(G_1)=\{u_1,\cdots,u_{n_1}\}$ and $V(G_2)=\{v_1,\cdots,v_{n_2}\}$,
where $n_1+n_2=n$. Clearly, $G[M]=2K_2$ or $G[M]=P_3$.

At first, we consider the case $G[M]=2K_2$. Without loss of
generality, let $M=\{u_1v_1,u_2v_2\}$. Since $n\geq 5$, $n_1\geq 3$
or $n_2\geq 3$. Without loss of generality, let $n_1\geq 3$.
Clearly, any two vertices $v_i,v_j\in V(G_2)$ are connected in
$\overline{G}$ since there exists a path $v_iu_3v_j$ in
$\overline{G}$. Furthermore, for any $u_i\in V(G_1)$, $u_iv_1\in
E(\overline{G})$ or $u_iv_2\in E(\overline{G})$. So $\overline{G}$
is connected and $\lambda(\overline{G})\geq 1$, a contradiction.

Next, we consider the case $G[M]=P_3$. Without loss of generality,
let $P=v_1u_1v_2$ be the path of order $3$. Since $n\geq 5$, there
exist at least two vertices in $G\setminus \{u_1,v_1,v_2\}$. If
$n_1\geq 2$ and $n_2\geq 3$, then we can check that $\overline{G}$
is connected, a contradiction. So we assume that $n_1=1$ or $n_2=2$,
that is, $V(G_2)=\{v_1,v_2\}$ or $V(G_1)=\{u_1\}$.

For the former, $V(G_1)=\{u_1,u_2,\cdots,u_{n-2}\}$. Since
$\lambda(G)=2$, $v_1v_2\in E(G)$. Clearly, $v_1u_j,v_2u_j\notin E(G)
\ (2\leq j\leq n-2)$, which implies that $v_1u_j,v_2u_j\in
E(\overline{G})$. Therefore, $u_1u_j\notin E(\overline{G}) \ (2\leq
j\leq n-2)$ since $\overline{G}$ is disconnected. Thus $u_1u_j\in
E(G)$ for each $j \ (2\leq j\leq n-2)$. So $d_G(u_1)=n-1$ and $G\in
\mathcal {G}_n^2$ (See Figure 1 $(b)$).

For the latter, let $V(G_2)=\{v_1,v_2,\cdots,v_{n-1}\}$. First we
consider the case $v_1v_2\in E(G)$. Since $u_1v_j\notin E(G) \
(3\leq j\leq n-1)$, we have $u_1v_j\in E(\overline{G})$. If $3\leq
d_G(v_1)\leq n-2$ and $3\leq d_G(v_2)\leq n-2$, then there exist two
vertices $v_i$ and $v_j$ such that $v_1v_i,v_2v_j\in E(\overline{G})
\ (3\leq i,j\leq n-1)$, which implies that $\overline{G}$ is
connected, a contradiction. So $d_G(v_1)= n-1$ or $d_G(v_2)=n-1$.
Without loss of generality, let $d_G(v_1)=n-1$. Thus $G\in \mathcal
{G}_n^3$ (See Figure 1 $(c)$). Now we focus on the graph $G\setminus
v_1$. Let $G_1,G_2,\cdots, G_{r}$ be the connected components of
$G\setminus v_1$ and $V(G_i)=\{v_{i1},v_{i2},\cdots,v_{in_i}\} \
(1\leq i\leq r)$, where $\sum_{i=1}^rn_i=n-1$. If there exists some
connected component $G_i$ such that $G_i=K_2$, then $G\in \mathcal
{G}_n^2$ (See Figure 1 $(b)$). So we assume $n_i\geq 3$. Then we
prove the following claim and get a contradiction.

\noindent {\bf Claim 1}. For each connected component $G_i$ of
$G\setminus v_1$, if $n_i\geq k$, or $n_i\leq k-1$ and $|E(G_i)|\geq
n_i$, then $\lambda_k(G)\geq 2$ for $3\leq k\leq n$.

\noindent{\itshape Proof of Claim $1$.} For an arbitrary $S\subseteq
V(G)$ with $|S|=k$, we only prove $\lambda(S)\geq 2$ for $v_1\notin
S$. The case for $v_1\in S$ can be proved similarly. If there exists
some connected component $G_i$ such that $S=V(G_i)$, then $n_i=k$
and $G_i$ has a spanning tree, say $T_i$. It is also a Steiner tree
connecting $S$. Since $T_i'=v_1v_{i1}\cup v_1v_{i2}\cdots \cup
v_1v_{in_i}$ is another Steiner tree connecting $S$ and $T_i, T_i'$
are two edge-disjoint trees, we have $\lambda(S)\geq 2$. Let us
assume now $S\neq V(G_i)$ for $n_i\geq k \ (1\leq i\leq r)$. Let
$S_i=S\cap V(G_i)\ (1\leq i\leq r)$ and $|S_i|=k_i$. Clearly,
$\bigcup_{i=1}^rS_i=S$ and $\sum_{i=1}^rk_i=k$. Thus $S_i\subset
V(G_i)$ for each connected component $G_i$ such that $n_i\geq k$,
and $S_j\subseteq V(G_j)$ for each connected component $G_j$ such
that $n_j\leq k-1$ and $|E(G_j)|\geq n_j$. We will show that there
are two edge-disjoint Steiner trees connecting $S_i\cup \{v_1\}$ in
$G[S_i\cup \{v_1\}]$ for each $i \ (1\leq i\leq r)$ so that we can
combine these trees to form two edge-disjoint Steiner trees
connecting $S$ in $G$. Suppose that $G_i$ is a connected component
such that $n_i\geq k$. Note that
$V(G_i)=\{v_{i1},v_{i2},\cdots,v_{in_i}\}$. Since $S_i\subset
V(G_i)$, there exists a vertex, without loss of generality, say
$v_{i1}$, such that $v_{i1}\notin S_i$. Clearly, $G_i$ contains a
spanning tree, say $T_{i1}'$. Thus $T_{i1}=v_1v_{i1}\cup T_{i1}'$ is
a Steiner tree connecting $S_i\cup \{v_1\}$ in $G[G_i\cup \{v_1\}]$.
Since $T_{i2}=v_1v_{i2}\cup v_1v_{i3}\cup \cdots \cup v_1v_{in_i}$
is another Steiner tree connecting $S_i\cup \{v_1\}$. Clearly,
$T_{i1}$ and $T_{i2}$ are edge-disjoint. Assume that $G_j$ is a
connected component such that $n_j\leq k-1$ and $|E(G_j)|\geq n_j$.
Note that $V(G_j)=\{v_{j1},v_{j2},\cdots,v_{jn_j}\}$. Then there
exists an edge, without loss of generality, say $e_j=v_{j1}v_{j2}\in
E(G_j)$ such that $G_j\setminus e_j$ contains a spanning tree of
$G_j$, say $T_{j1}'$. Thus $T_{j1}=v_1v_{j1}\cup T_{j1}'$ and
$T_{j2}=v_{j1}v_{j2}\cup v_{1}v_{j2}\cup \cdots \cup v_{1}v_{jn_j}$
are two edge-disjoint Steiner trees connecting $S_j\cup \{v_{1}\}$.
Now we combine these small trees connecting $S_i\cup \{v_1\} \
(1\leq i\leq r)$ by the vertex $v_1$ to form two big trees
connecting $S$. Clearly, $T_1=T_{11}\cup T_{21}\cup \cdots \cup
T_{r1}$ and $T_2=T_{12}\cup T_{22}\cup \cdots \cup T_{r2}$ are our
desired trees, that is, $\lambda(S)\geq 2$. From the arbitrariness
of $S$, we have $\lambda_k(G)\geq 2$.\hfill$\Box$\vspace{10pt}

By Claim $1$, we know that $G\in \mathcal {G}_n^3$ and there exists
a connected component $G_i$ of $G\setminus \{v_1\}$ such that
$n_i\leq k-1$ and $G_i$ is a tree.

We next consider the case $v_1v_2\notin E(G)$ (See Figure 1 $(d)$).
Thus $v_1v_2\in E(\overline{G})$. Since $u_1v_j\notin E(G) \ (3\leq
j\leq n-1)$, $u_1v_j\in E(\overline{G})$, which results in
$v_1v_j,v_2v_j\notin E(\overline{G})$ since $\overline{G}$ is
disconnected. Thus $v_1v_j,v_2v_j\in E(G)$ for each $j \ (3\leq
j\leq n-1)$. Let $R=\{v_j|3\leq j\leq n-1\}$. If $|E(G[R])|\geq 2$,
then $G$ contains a subgraph $K_{2,n-2}^{++}$, which implies that
$\lambda_n(G)\geq 2$ by $(1)$ of Observation \ref{obs3}. Combining
this with Proposition \ref{pro2}, $\lambda_k(G)\geq 2$ for $3\leq
k\leq n$, a contradiction. If $|E(G[R])|<2$, then $G=K_{2,n-2}$ and
$K_{2,n-2}^{+}$. From Observation \ref{obs3} and Proposition
\ref{pro2}, we have $\lambda_k(K_{2,n-2}^{+})\geq 2$ for $3\leq
k\leq n-1$ and $\lambda_k(K_{2,n-2})\geq 2$ for $3\leq k\leq n-2$, a
contradiction. So $G=K_{2,n-2}^{+}$ for $k=n$, or $G=K_{2,n-2}$ for
$k=n$, or $G=K_{2,n-2}$ for $k=n-1$.

\emph{Case 3.}~~$\lambda(G)=3$.

For $n=4$, $G=K_4$, $\lambda_3(G)=\lambda_4(G)=2$ by Lemma
\ref{lem3}. Then $\lambda_k(G)\geq 2$, a contradiction. Assume
$n\geq 5$.
\begin{figure}[h,t,b,p]
\begin{center}
\scalebox{0.8}[0.8]{\includegraphics{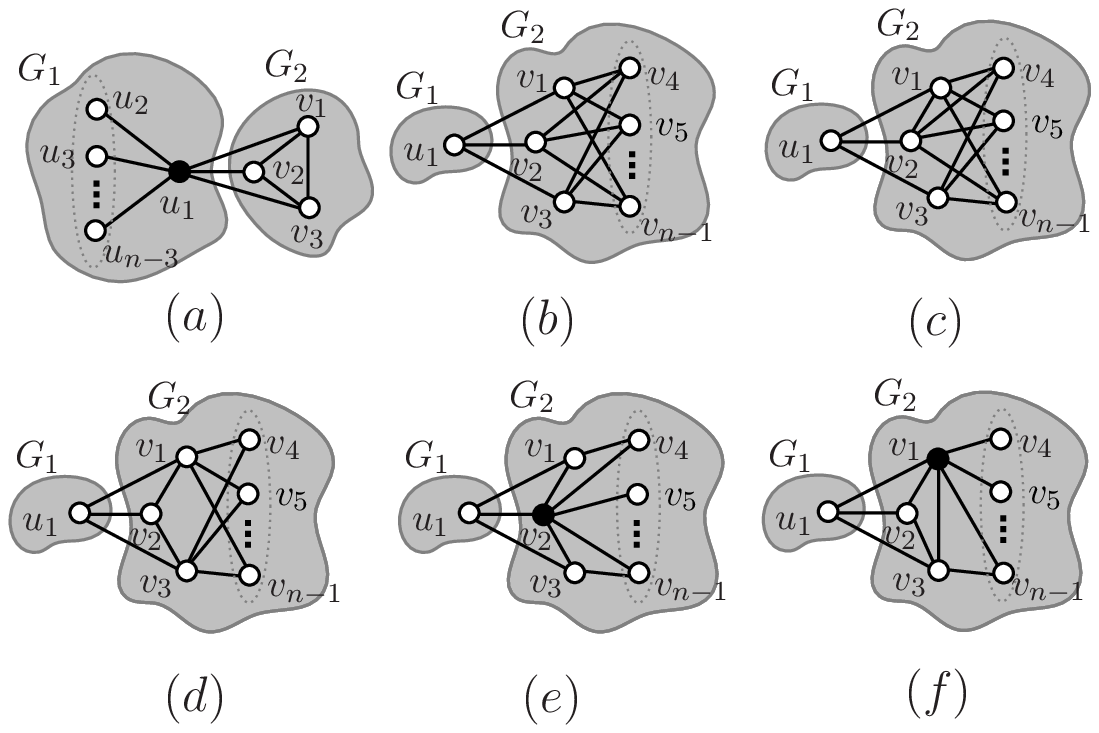}}\\
Figure 3. Graphs for Case 3 of Proposition \ref{pro3}.
\end{center}
\end{figure}
Since $\lambda(G)=3$, there exists an edge cut $M$ such that
$|M|=3$. Let $G_1$ and $G_2$ be two connected components of
$G\setminus M$, $V(G_1)=\{u_1,u_2,\cdots,u_{n_1}\}$ and
$V(G_2)=\{v_1,v_2,\cdots,v_{n_2}\}$, where $n_1+n_2=n$. Clearly,
$G[M]=P_4$ or $G[M]=P_3\cup K_2$ or $G[M]=3K_2$ or $G[M]=K_{1,n-3}$.
For the former three cases, $n_i\geq 3 \ (i=1,2)$ and $n\geq 6$
since $\lambda(G)=3$. To shorten the discussion, we only prove
$\lambda(\overline{G})\geq 1$ for $G[M]=P_4$ and get a contradiction
among the former three cases. Without loss of generality, let
$G[M]=P_4=u_1v_1u_2v_2$. For any $u_i,u_j\in V(G_1) \ (1\leq i\leq
n_1)$, $u_i$ and $u_j$ are connected in $\overline{G}$ since there
exists a path $u_iv_3u_j$ in $\overline{G}$; for any $v_i,v_j\in
V(G_2) \ (1\leq i\leq n_2)$, $v_i$ and $v_j$ are connected in
$\overline{G}$ since there exists a path $v_iu_3v_j$ in
$\overline{G}$; for any $u_i\in V(G_1)$ and $v_j\in V(G_2)$($i\neq
3$ and $j\neq 3$), $u_i$ and $u_j$ are connected in $\overline{G}$
since there exists a path $u_iv_3u_3v_j$ in $\overline{G}$. Since
$u_3v_j\in E(\overline{G}) \ (1\leq j\leq n_2)$ and $v_3u_i\in
E(\overline{G}) \ (1\leq i\leq n_1)$, $\overline{G}$ is connected, a
contradiction.

Now we consider the graph $G$ such that $G[M]=K_{1,n-3}$. Assume
$n_1\geq 2$. If $n_2\geq 4$, then we can check that $\overline{G}$
is connected and get a contradiction. Therefore, $n_2=3$,
$V(G_2)=\{v_1,v_2,v_3\}$ and $V(G_1)=\{u_1,u_2\cdots, u_{n-3}\}$.
Since $\lambda(G)=3$, it follows that $v_1v_2,v_2v_3,v_1v_3\in
E(G)$. Since $v_iu_j\notin E(G) \ (1\leq i\leq 3, \ 2\leq j\leq
n-3)$, we have $v_iu_j\in E(\overline{G})$. If there exists some
vertex $u_j \ (2\leq j\leq n-3)$ such that $u_1u_j\in
E(\overline{G})$, then $\overline{G}$ is connected, a contradiction.
So $u_1u_j\in E(G)$ for $2\leq j\leq n-3$. Thus $d_G(u_1)=n-1$ (See
Figure 3 $(a)$). From Lemma \ref{lem5}, $\lambda_k(G)\geq 2$ for
$3\leq k\leq n$ since $\lambda(G)=3$, a contradiction.

Let us now assume $n_1=1$. Then $V(G_1)=\{u_1\}$ and
$V(G_2)=\{v_1,v_2\cdots, v_{n-1}\}$. If $G[\{v_1,v_2,v_3\}]=3K_1$ or
$G[\{v_1,v_2,v_3\}]=2K_1\cup K_2$, then we have $u_1v_j\in
E(\overline{G})$ since $u_1v_j\notin E(G) \ (4\leq j\leq n-1)$. From
this together with the fact that $\overline{G}$ is disconnected and
$v_1v_3,v_2v_3\in E(\overline{G})$, $v_iv_j\notin E(\overline{G}) \
(1\leq i\leq 3, \ 4\leq j\leq n-1)$, we have that $v_iv_j\in E(G) \
(1\leq i\leq 3, \ 4\leq j\leq n-1)$. Thus $G$ contains a complete
bipartite graph $K_{3,n-3}$ as its subgraph (See Figure 3 $(b)$ and
$(c)$). From $(1)$ of Lemma \ref{lem1},
$\lambda_n(G)=\lfloor\frac{3(n-3)}{n-1}\rfloor\geq 2$ for $n\geq 7$,
which implies $\lambda_k(G)\geq 2$ for $3\leq k\leq n$ and $n\geq
7$. Since $\lambda(G)=3$, $n\geq 6$. So we only need to consider the
case $n=6$. Thus $G=H_i \ (1\leq i\leq 4)$ (See Figure 4). If $G=H_i
\ (2\leq i\leq 4)$, then $\lambda_n(G)\geq 2$ for $k=n=6$ (See
Figure 4 $(b),(c),(d)$). Therefore $\lambda_k(G)\geq 2$ for $3\leq
k\leq 6$. If $G=H_1$, then $\lambda_n(G)\leq
\lfloor\frac{|E(G)|}{n-1}\rfloor=\lfloor\frac{9}{5}\rfloor=1$ for
$k=n=6$. For $k=5$, we can check that $\lambda_3(G)\geq
\lambda_4(G)\geq \lambda_5(G)\geq 2$ (See Figure 4 $(e)$). So
$G=K_{3,3}$ for $k=n=6$.

\begin{figure}[h,t,b,p]
\begin{center}
\scalebox{0.9}[0.9]{\includegraphics{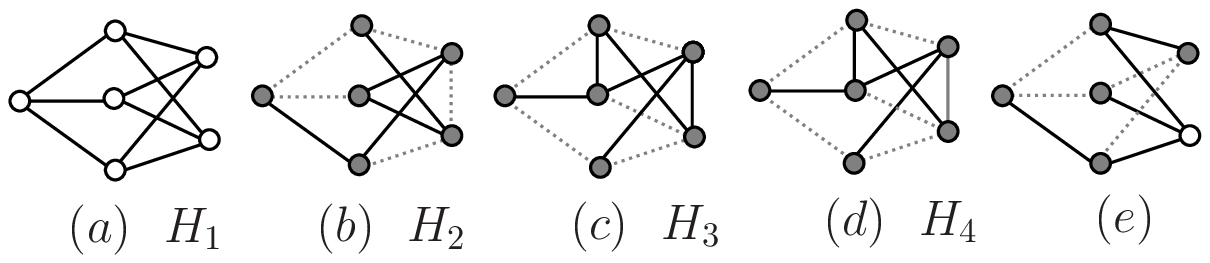}}\\
Figure 4. Graphs for Case 3 of Proposition \ref{pro3}.
\end{center}
\end{figure}

Suppose $G[\{v_1,v_2,v_3\}]=P_3$. Without loss of generality, let
$v_1v_2,v_2v_3\in E(G)$. If $3\leq d_G(v_2)\leq n-2$ (See Figure 3
$(d)$), then there exists at least one vertex $v_j$ such that
$v_2v_j\in E(\overline{G})$, which results in $v_1v_j,v_3v_j\notin
E(\overline{G}) \ (4\leq j\leq n-1)$ since $u_1v_j\in
E(\overline{G}) \ (4\leq j\leq n-1)$, $v_1v_3\in E(\overline{G})$
and $\overline{G}$ is disconnected. Thus $v_1v_j,v_3v_j\in E(G)$ for
each $j \ (4\leq j\leq n-1)$. Since $d(v_4)\geq \delta(G)\geq
\lambda(G)=3$, we have $v_4v_2\in E(G)$ or there exists some vertex
$v_j \ (5\leq j\leq n-1)$ such that $v_4v_j\in E(G)$, which implies
that $G$ contains a subgraph $K_{2,n-2}^{++}$ and so
$\lambda_n(G)\geq 2$ by $(1)$ of Observation \ref{obs3}. From
Proposition \ref{pro2}, $\lambda_k(G)\geq 2$ for $3\leq k\leq n$, a
contradiction. If $d_G(v_2)=n-1$ (See Figure 3 $(e)$), then
$\lambda_k(G)\geq 2$ for $3\leq k\leq n$ by Lemma \ref{lem5} since
$\lambda(G)=3$, a contradiction.

Suppose $G[\{v_1,v_2,v_3\}]=K_3$. Without loss of generality, let
$v_1v_2,v_1v_3,v_2v_3\in E(G)$. If $d_G(v_1)=n-1$ or $d_G(v_2)=n-1$
or $d_G(v_3)=n-1$ (See Figure 3 $(f)$), then by Lemma \ref{lem5}
$\lambda_k(G)\geq 2$ for $3\leq k\leq n$ since $\lambda(G)=3$, a
contradiction. If $3\leq d_G(v_i)\leq n-2(1\leq i\leq 3)$, then
$\overline{G}$ is connected, a contradiction.
\end{proof}

We now investigate the upper bounds of
$\lambda_k(G)+\lambda_k(\overline{G})$ and $\lambda_k(G)\cdot
\lambda_k(\overline{G})$.

\begin{lemma}\label{lem6}
Let $G\in \mathcal {G}(n)$. Then

$(1)$ $\lambda_k(G)+\lambda_k(\overline{G})\leq n-\lceil k/2
\rceil$;

$(2)$ $\lambda_k(G)\cdot \lambda_k(\overline{G})\leq
\big[\frac{n-\lceil k/2 \rceil}{2}\big]^2$.

Moreover, the two upper bounds are sharp.
\end{lemma}

\begin{proof}
$(1)$ Since $G\cup \overline{G}=K_n$,
$\lambda_k(G)+\lambda_k(\overline{G})\leq \lambda_k(K_n)$. Combining
this with Lemma \ref{lem3},
$\lambda_k(G)+\lambda_k(\overline{G})\leq
n-\lceil\frac{k}{2}\rceil$.

$(2)$ The conclusion holds by $(1)$.
\end{proof}

Let us focus on $(1)$ of Lemma \ref{lem6}. If one of $G$ and
$\overline{G}$ is disconnected, we can characterize the graphs
attaining the upper bound by Lemma \ref{lem4}.

\begin{proposition}\label{pro4}
For any graph $G$ of order $n$, if $G$ is disconnected, then
$\lambda_k(G)+ \lambda_k(\overline{G})=n-\lceil\frac{k}{2}\rceil$ if
and only if $\overline{G}=K_n$ for $k$ even;
$\overline{G}=K_n\setminus M$ for $k$ odd, where $M$ is an edge set
such that $0\leq |M|\leq \frac{k-1}{2}$.
\end{proposition}

If both $G$ and $\overline{G}$ are all connected, we can obtain a
structural property of the graphs attaining the upper bound although
it seems too difficult to characterize them.

\begin{proposition}\label{pro5}
If $\lambda_k(G)+\lambda_k(\overline{G})=n-\lceil\frac{k}{2}\rceil$,
then $\Delta(G)-\delta(G)\leq \lceil\frac{k}{2}\rceil-1$.
\end{proposition}

\begin{proof}
Assume that $\Delta(G)-\delta(G)\geq \lceil\frac{k}{2}\rceil$. Since
$\lambda_k(\overline{G})\leq \delta(\overline{G})=n-1-\Delta(G)$,
$\lambda_k(G)+\lambda_k(\overline{G})\leq
\delta(G)+n-1-\Delta(G)\leq n-1-\lceil\frac{k}{2}\rceil$, a
contradiction.
\end{proof}

One can see that the graphs with
$\lambda_k(G)+\lambda_k(\overline{G})=n-\lceil\frac{k}{2}\rceil$
must have a uniform degree distribution. Actually, we can construct
a graph class to show that the two upper bounds of Lemma \ref{lem6}
are tight for $k=n$.

\textbf{Example 2.}~~Let $n,r$ be two positive integers such that
$n=4r+1$. From $(1)$ of Lemma \ref{lem1}, we know that the $STP$
number of the complete bipartite graph $K_{2r,2r+1}$ is
$\lfloor\frac{2r(2r+1)}{2r+(2r+1)-1}\rfloor=r$, that is,
$\lambda_{n}(K_{2r,2r+1})=r$. Let $\mathcal {E}$ be the set of the
edges of these $r$ spanning trees in $K_{2r,2r+1}$. Then there exist
$2r(2r+1)-4r^2=2r$ remaining edges in $K_{2r,2r+1}$ except the edges
in $\mathcal {E}$. Let $M$ be the set of these $2r$ edges. Set
$G=K_{2r,2r+1}\setminus M$. Then $\lambda_{n}(G)=r$, $M\subseteq
E(\overline{G})$ and $\overline{G}$ is a graph obtained from two
cliques $K_{2r}$ and $K_{2r+1}$ by adding $2r$ edges in $M$ between
them, that is, one endpoint of each edge belongs to $K_{2r}$ and the
other endpoint belongs to $K_{2r+1}$. Note that
$E(\overline{G})=E(K_{2r})\cup M\cup E(K_{2r+1})$. Now we show that
$\lambda_{n}(\overline{G})\geq r$. As we know, $K_{2r}$ contains $r$
Hamiltonian paths, say $P_{1},P_{2},\cdots,P_{r}$, and so does
$K_{2r+1}$, say $P_{1}',P_{2}',\cdots,P_{r}'$. Pick up $r$ edges
from $M$, say $e_1,e_2,\cdots,e_r$, let $T_i=P_{i}\cup P_{i}'\cup
e_i(1\leq i\leq r)$. Then $T_{1},T_{2},\cdots,T_{r}$ are $r$
spanning trees in $\overline{G}$, namely,
$\lambda_{n}(\overline{G})\geq r$. Since
$|E(\overline{G})|={{2r}\choose{2}}+{{2r+1}\choose{2}}+2r=4r^2+2r$
and each spanning tree uses $4r$ edges, these edges can form at most
$\lfloor\frac{4r^2+2r}{4r}\rfloor=r$ spanning trees, that is,
$\lambda_{n}(\overline{G})\leq r$. So $\lambda_{n}(\overline{G})=r$.

Clearly, $\lambda_{n}(G)+\lambda_{n}(\overline{G})=2r=\frac{n-1}{2}
=n-\lceil\frac{n}{2}\rceil$ and $\lambda_{n}(\overline{G})\cdot
\lambda_{n}(\overline{G})=r^2=\big[\frac{n-\lceil n/2
\rceil}{2}\big]^2$, which implies that the upper bound of Lemma
\ref{lem6} is sharp.

Combining Lemmas \ref{lem2} and \ref{lem6}, we give our main result.

\begin{theorem}\label{th3}
Let $G\in \mathcal {G}(n)$. Then

$(1)$ $1\leq \lambda_k(G)+\lambda_k(\overline{G})\leq n-\lceil k/2
\rceil$;

$(2)$ $0\leq \lambda_k(G)\cdot \lambda_k(\overline{G})\leq
\big[\frac{n-\lceil k/2 \rceil}{2}\big]^2$.

Moreover, the upper and lower bounds are sharp.
\end{theorem}

\section{Nordhaus-Gaddum-type results in $\mathcal {G}(n,m)$}

Achthan et. al. \cite{Achuthan} restricted their attention to the
subclass of $\mathcal {G}(n,m)$ consisting of graphs with exactly
$m$ edges. They investigated the edge-connectivity, diameter and
chromatic number parameters. For edge-connectivity $\lambda(G)$,
they showed that $\lambda(G)+\lambda(\overline{G})\geq
max\{1,n-1-m\}$. In this section, we consider a similar problem on
the generalized edge-connectivity.

\begin{lemma}\label{lem7}
If $M$ is an edge set of the complete graph $K_n$ such that $0\leq
m\leq \lfloor\frac{n}{3}\rfloor$ where $|M|=m$, then $G=K_n\setminus
M$ contains $\ell$ edge-disjoint spanning trees, where $\ell=min \{
n-2m-1, \lfloor \frac{n}{2}-\frac{2m}{n-1} \rfloor \}$.
\end{lemma}

\begin{proof}
Let $\mathscr{P}=\bigcup_{i=1}^pV_i$ be a partition of $V(G)$ with
$|V_i|=n_i \ (1\leq i\leq p)$, and $\mathcal {E}_p$ be the set of
edges between distinct parts of $\mathscr{P}$ in $G$. It suffices to
show that $|\mathcal {E}_p|\geq \ell(|\mathscr{P}|-1)$ so that we
can use Nash-Williams-Tutte Theorem.

The case $p=1$ is trivial, thus we assume $2\leq p\leq n$. Then
$|\mathcal {E}_p|\geq
{{n}\choose{2}}-\sum_{i=1}^p{{n_i}\choose{2}}-|M|\geq
{{n}\choose{2}}-\sum_{i=1}^p{{n_i}\choose{2}}-m$. We will show that
${{n}\choose{2}}-\sum_{i=1}^p{{n_i}\choose{2}}-m\geq \ell(p-1)$,
that is, $\frac{n(n-1)}{2}-m-\ell(p-1)\geq
\sum_{i=1}^p{{n_i}\choose{2}}$. We only need to prove that
$\frac{n(n-1)}{2}-m-\ell(p-1)\geq
max\{\sum_{i=1}^p{{n_i}\choose{2}}\}$. Since
$f(n_1,n_2,\cdots,n_p)=\sum_{i=1}^p{{n_i}\choose{2}}$ achieves its
maximum value when $n_1=n_2=\cdots=n_{p-1}=1$ and $n_p=n-p+1$, we
need the inequality $\frac{n(n-1)}{2}-m-\ell(p-1)\geq
{{1}\choose{2}}(p-1)+{{n-p+1}\choose{2}}$, that is,
$\frac{n(n-1)}{2}-m-\frac{(n-p+1)(n-p)}{2}\geq \ell(p-1)$. Actually,
$\ell\leq \frac{n(n-1)-(n-p+1)(n-p)-2m}{2(p-1)}$ is our required
inequality, namely, $\ell\leq
n-\frac{1}{2}-(\frac{p-1}{2}+\frac{2m}{p-1})$. Since
$f(x)=\frac{x}{2}+\frac{2m}{x}$ achieves its maximum value $max
\{2m+\frac{1}{2}, \frac{n-1}{2}+\frac{2m}{n-1}\}$ when $1\leq x\leq
n-1$, we need $\ell\leq min \{n-2m-1, \frac{n}{2}-\frac{2m}{n-1}\}$.
Since this inequality holds for $0\leq m\leq
\lfloor\frac{n}{3}\rfloor$, we have $|\mathcal {E}_p|\geq
{{n}\choose{2}}-\sum_{i=1}^p{{n_i}\choose{2}}-|M|\geq \ell(p-1)$.
From Theorem \ref{th1}, we know that $G$ has $\ell$ edge-disjoint
spanning trees.
\end{proof}

\begin{lemma}\label{lem8}
Let $G\in \mathcal {G}(n,m)$. For $n\geq 6$, we have

$(1)$ $\lambda_k(G)+\lambda_k(\overline{G})\geq L(n,m)$, where

$$L(n,m)=\left\{
\begin{array}{cc}
max\{1,\lfloor\frac{1}{2}(n-2-m)\rfloor\} &~~~~~~~~if~\lfloor\frac{n}{3}\rfloor+1\leq m\leq {{n}\choose{2}},\\
min \{n-2m-1, \lfloor \frac{n}{2}-\frac{2m}{n-1}\rfloor\}&if~0\leq
m\leq \lfloor\frac{n}{3}\rfloor.
\end{array}
\right.$$

$(2)$ $\lambda_k(G)\cdot \lambda_k(\overline{G})\geq 0$.
\end{lemma}

\begin{proof}
$(1)$ Since at least one of $G$ and $\overline{G}$ must be
connected, we have $\lambda_k(G)+\lambda_k(\overline{G})\geq 1$. For
$m<n-1$, $\lambda_k(G)+\lambda_k(\overline{G})\geq
\lfloor\frac{1}{2}\lambda(G)\rfloor+\lfloor\frac{1}{2}\lambda(\overline{G})\rfloor
\geq
\lfloor\frac{1}{2}(\lambda(G)+\lambda(\overline{G})-1)\rfloor\geq
\lfloor\frac{1}{2}(max\{1,n-1-m\}-1)\rfloor\geq
\lfloor\frac{1}{2}(n-2-m)\rfloor$  by Proposition \ref{pro1}. So
$\lambda_k(G)+\lambda_k(\overline{G})\geq
max\{1,\lfloor\frac{1}{2}(n-2-m)\rfloor\}$. In particular, for
$0\leq m\leq \lfloor\frac{n}{3}\rfloor$, we can give a better lower
bound of $\lambda_k(G)+\lambda_k(\overline{G})$  by Lemma
\ref{lem7}, that is,
$\lambda_k(G)+\lambda_k(\overline{G})=\lambda_k(\overline{G})\geq
\lambda_n(\overline{G})\geq min \{n-2m-1, \lfloor
\frac{n}{2}-\frac{2m}{n-1}\rfloor\}$.

To show the sharpness of the above lower bound for
$\lfloor\frac{n}{3}\rfloor+1\leq m\leq {{n}\choose{2}}$, we consider
the graph $G=K_{1,n-2}\cup K_1$. Then $m=n-2$ and $\overline{G}$ is
a graph obtained from a complete graph $K_{n-1}$ by attaching a
pendant edge. Clearly, $\lambda_k(G)=0$ and
$\lambda_k(\overline{G})=1$. So
$\lambda_k(G)+\lambda_k(\overline{G})=1=max\{1,\lfloor\frac{1}{2}(n-2-m)\rfloor\}$.
To show the sharpness of the above lower bound for $0\leq m\leq
\lfloor\frac{n}{3}\rfloor$, we consider the graph $G=nK_1$. Thus
$m=0$ and $\overline{G}=K_n$. Since
$\lambda_n(G)+\lambda_n(\overline{G})=0+\lfloor\frac{n}{2}\rfloor=min
\{n-2\cdot 0-1, \lfloor \frac{n}{2}-\frac{2\cdot 0}{n-1}\rfloor\}$,
that is, the lower bound is sharp for $k=n$.

$(2)$ The inequality follows from Theorem \ref{th3}.
\end{proof}

It was pointed out by Harary \cite{Harary} that given the number of
vertices and edges of a graph, the largest connectivity possible can
also be read out of the inequality $\kappa(G)\leq \lambda(G)\leq
\delta(G)$.

\begin{theorem}\cite{Harary}\label{th2}
For each $n,m$ with $0\leq n-1\leq m\leq {{n}\choose{2}}$,

$$\kappa(G)\leq \lambda(G)\leq \Big\lfloor \frac{2m}{n}\Big\rfloor,$$
where the maximum are taken over all graphs $G\in \mathcal
{G}(n,m)$.
\end{theorem}

Now we will study a similar problem for the generalized
edge-connectivity, which will be used in $(2)$ of Lemma \ref{lem9}.

\begin{corollary}\label{cor2}
For any graph $G\in \mathcal {G}(n,m)$ and $3\leq k\leq n$,
$\lambda_k(G)=0$ for $m<n-1$; $\lambda_k(G)\leq
\lfloor\frac{2m}{n}\rfloor$ for $m\geq n-1$.
\end{corollary}

\begin{proof}
Let $G\in \mathcal {G}(n,m)$. When $0\leq m<n-1$, $G$ must be
disconnected and hence $\lambda_k(G)=0$. If $m\geq n-1$,
$\lambda_k(G)\leq \lambda(G)\leq \lfloor\frac{2m}{n}\rfloor$ by
$(1)$ of Observation \ref{obs1} and Theorem \ref{th2}.
\end{proof}

Although the above bound of $\lambda_k(G)$ is the same as
$\lambda(G)$, the graphs attaining the upper bound seems to be very
rare. Actually, we can obtain some structural properties of these
graphs.

\begin{proposition}\label{pro6}
For any $G\in \mathcal {G}(n,m)$ and $3\leq k\leq n$, if
$\lambda_k(G)=\lfloor\frac{2m}{n}\rfloor$ for $m\geq n-1$, then

$(1)$ $\frac{2m}{n}$ is not an integer;

$(2)$ $\delta(G)=\lfloor\frac{2m}{n}\rfloor$;

$(3)$ for $u,v\in V(G)$ such that
$d_G(u)=d_G(v)=\lfloor\frac{2m}{n}\rfloor$, $uv\notin E(G)$.
\end{proposition}

\begin{proof}
One can check that the conclusion holds for the case $m=n-1$. Assume
$m\geq n$. We claim that $\frac{2m}{n}$ is not an integer.
Otherwise, let $r=\frac{2m}{n}$ be an integer. We will show that
$\lambda_k(G)\leq r-1=\frac{2m}{n}-1$ and get a contradiction. If
$G$ has at least one vertex $v_i$ such that $d(v_i)>r$, then, since
the average degree of $G$ is exactly $r$, there must be a vertex
$v_j$ whose degree $d(v_j)<r$. From $(1)$ of Observation \ref{obs1},
we have $\lambda_k(G)\leq \delta(G)\leq d(v_j)<r$, that is,
$\lambda_k(G)\leq r-1$. If, on the other hand, $G$ is a regular
graph, then by $(3)$ of Observation \ref{obs1}, $\lambda_k(G)\leq
\delta(G)-1=r-1$. So $(1)$ holds.

For a graph G such that $\frac{2m}{n}$ is not an integer,
$\lfloor\frac{2m}{n}\rfloor=\lambda_k(G)\leq \delta(G)\leq
\lfloor\frac{2m}{n}\rfloor$, that is,
$\delta(G)=\lfloor\frac{2m}{n}\rfloor$. So $(2)$ holds.

For $u,v\in V(G)$ such that
$d_G(u)=d_G(v)=\lfloor\frac{2m}{n}\rfloor$, we claim that $uv\notin
E(G)$. Otherwise, $uv\in E(G)$. Since
$d_G(u)=d_G(v)=\delta(G)=\lfloor\frac{2m}{n}\rfloor$,
$\lambda_k(G)\leq \delta(G)-1=\lfloor\frac{2m}{n}\rfloor-1$ by $(3)$
of Observation \ref{obs1}, a contradiction. So $(3)$ holds.
\end{proof}

\begin{corollary}\label{cor3}
For any graph $G$ of order $n$ and size $m$, if $\frac{2m}{n}$ is an
integer, then $\lambda_k(G)\leq \frac{2m}{n}-1$.
\end{corollary}

\begin{lemma}\label{lem9}
Let $G\in \mathcal {G}(n,m)$. Then

$(1)$ $\lambda_k(G)+\lambda_k(\overline{G})\leq M(n,m)$, where

$$M(n,m)=\left\{
\begin{array}{cc}
n-\lceil\frac{k}{2}\rceil &if~m\geq n-1,~~~~~~~~~~~~~~~~~~~~~~~~~~~~~~~~~~~\\
&~or~k~is~even~and~m=0,~~~~~~~~~~~~\\
&~or~k~is~odd~and~0\leq m\leq \frac{k-1}{2};~~~~~\\
n-\lceil\frac{k}{2}\rceil-1~~~~~ &if~k~is~even~and~1\leq m<
n-1,~~~~~~~~~~~\\
&~or~k~is~odd~and~\frac{k+1}{2}\leq m<n-1.
\end{array}
\right.$$

$(2)$ $\lambda_k(G)\cdot \lambda_k(\overline{G})\leq N(n,m)$, where

$$N(n,m)=\left\{
\begin{array}{cc}
0 &if~0\leq m\leq n-2~~~~~~~~~~~~~~~~~~~~,\\
(\frac{2m}{n}-1)(n-2-\frac{2m}{n})&if~m\geq n-1~and~2m\equiv
0(mod~n),\\
\lfloor\frac{2m}{n}\rfloor(n-2-\lfloor\frac{2m}{n}\rfloor)&otherwise.~~~~~~~~~~~~~~~~~~~~~~~~~~~~~.
\end{array}
\right.$$

Moreover, these upper bounds are sharp.
\end{lemma}

\begin{proof}
From Theorem \ref{th3}, $(1)$ holds for $m\geq n-1$. We have given a
graph class to show that the upper bound is sharp. From Proposition
\ref{pro4}, $\lambda_k(G)+\lambda_k(\overline{G})=
\lambda_k(\overline{G})=n-\lceil\frac{k}{2}\rceil$ for $k$ even and
$m=0$, or $k$ odd and $0\leq m\leq \frac{k-1}{2}$. So for $k$ even
and $1\leq m< n-1$, or $k$ odd and $\frac{k+1}{2}\leq m< n-1$,
$\lambda_k(G)+\lambda_k(\overline{G})\leq
n-\lceil\frac{k}{2}\rceil-1$.

To prove the sharpness of the bound for $k$ odd and
$\frac{k+1}{2}\leq m<n-1$, we consider the graph
$G=K_{1,\frac{k+1}{2}}\cup (n-\frac{k+3}{2})K_1$. Now $\overline{G}$
is a graph obtained from the complete graph $K_n$ by deleting all
the edges of a star $K_{1,\frac{k+1}{2}}$. On one hand, by Lemma
\ref{lem4}, $\lambda_k(\overline{G})\leq n-\frac{k+1}{2}-1$. On the
other hand, by Lemma \ref{lem4}, we have $\lambda_k(\overline{G}+e)=
n-\frac{k+1}{2}$ for any $e\notin E(\overline{G})$, which implies
that $\lambda_k(\overline{G})\geq n-\frac{k+1}{2}-1$ (Note that
$\lambda_k(H\setminus e)\geq \lambda_k(H)-1$  for a connected graph
$H$, where $e\in E(H)$). So
$\lambda_k(G)+\lambda_k(\overline{G})=\lambda_k(\overline{G})=n-\frac{k+1}{2}-1$.
By the same reason, for $k$ even and $1\leq m<n-1$ one can check
that the graph $G=K_2\cup (n-2)K_1$ satisfies that
$\lambda_k(G)+\lambda_k(\overline{G})=\lambda_k(\overline{G})\geq
n-\frac{k}{2}-1$.

$(2)$ First, if $0\leq m\leq n-2$, then $G\in \mathcal {G}(n,m)$ is
disconnected. So $\lambda_k(G)\cdot \lambda_k(\overline{G})=0$. Next
if $\frac{2m}{n}=r$ is an integer, then
$\frac{2e(\overline{G})}{n}=n-1-r$ is also an integer. From
Corollary \ref{cor3}, we have $\lambda_k(G)\leq r-1$ and
$\lambda_k(\overline{G})\leq n-2-r$. So $\lambda_k(G)\cdot
\lambda_k(\overline{G})\leq
(r-1)(n-2-r)=(\frac{2m}{n}-1)(n-2-\frac{2m}{n})$. Finally, if
$2m=nr+\ell$ where $1\leq \ell\leq n-1$, then $\Delta(G)\geq r+1$.
By $(1)$ of Observation \ref{obs1}, $\lambda_k(\overline{G})\leq
\delta(\overline{G})=n-1-\Delta(G)\leq n-2-r$. So $\lambda_k(G)\cdot
\lambda_k(\overline{G})\leq
r(n-2-r)=\lfloor\frac{2m}{n}\rfloor(n-2-\lfloor\frac{2m}{n}\rfloor)$.
\end{proof}

To show the sharpness of the upper bound for $m\geq n-1$ and
$2m\equiv 0 \ (mod~n)$, we consider the following example.

\textbf{Example 3.}~~Let $G$ be a cycle $C_n=w_1w_2\cdots
w_nw_1(n\geq 9)$. Since $\frac{2m}{n}=2$ is an integer,
$\lambda_3(G)=\frac{2m}{n}-1=1$. It suffices to prove that
$\lambda_3(\overline{G})=n-2-\frac{2m}{n}=n-4$.

Choose $S=\{x,y,z\}\subseteq V(C_n)=V(G)$. We will show that
$\lambda(S)\geq n-4$. If $d_{C_n}(x,y)=1$ and $d_{C_n}(y,z)=1$,
without loss of generality, let $N_{C_n}(x)=\{x_1,y\}$ and
$N_{C_n}(z)=\{y,z_2\}$, then the trees $T_i= xw_i\cup yw_i\cup zw_i$
together with $T_1= xz\cup zx_1\cup x_1y$ form $n-4$ edge-disjoint
$S$-trees (See Figure 5 $(a)$), namely, $\lambda(S)\geq n-4$, where
$\{w_1,w_2,\cdots,w_{n-5}\}=V(G)\setminus \{x,y,z,x_1,z_2\}$.

If $d_{C_n}(x,y)=2$ and $d_{C_n}(y,z)=1$, without loss of
generality, let $N_{C_n}(x)=\{x_1,y_1\}$ and $N_{C_n}(z)=\{y_1,z\}$
and $N_{C_n}(z)=\{y,z_2\}$, then the trees $T_i= xw_i\cup yw_i\cup
zw_i$ together with $T_1= xy\cup xz$ and $T_2= z_2x\cup z_2y\cup
z_2y_1\cup y_1z$ form $n-4$ edge-disjoint $S$-trees (See Figure 5
$(b)$), namely, $\lambda(S)\geq n-4$, where
$\{w_1,w_2,\cdots,w_{n-6}\}=V(G)\setminus \{x,y,z,x_1,y_1,z_2\}$.

\begin{figure}[h,t,b,p]
\begin{center}
\scalebox{0.8}[0.8]{\includegraphics{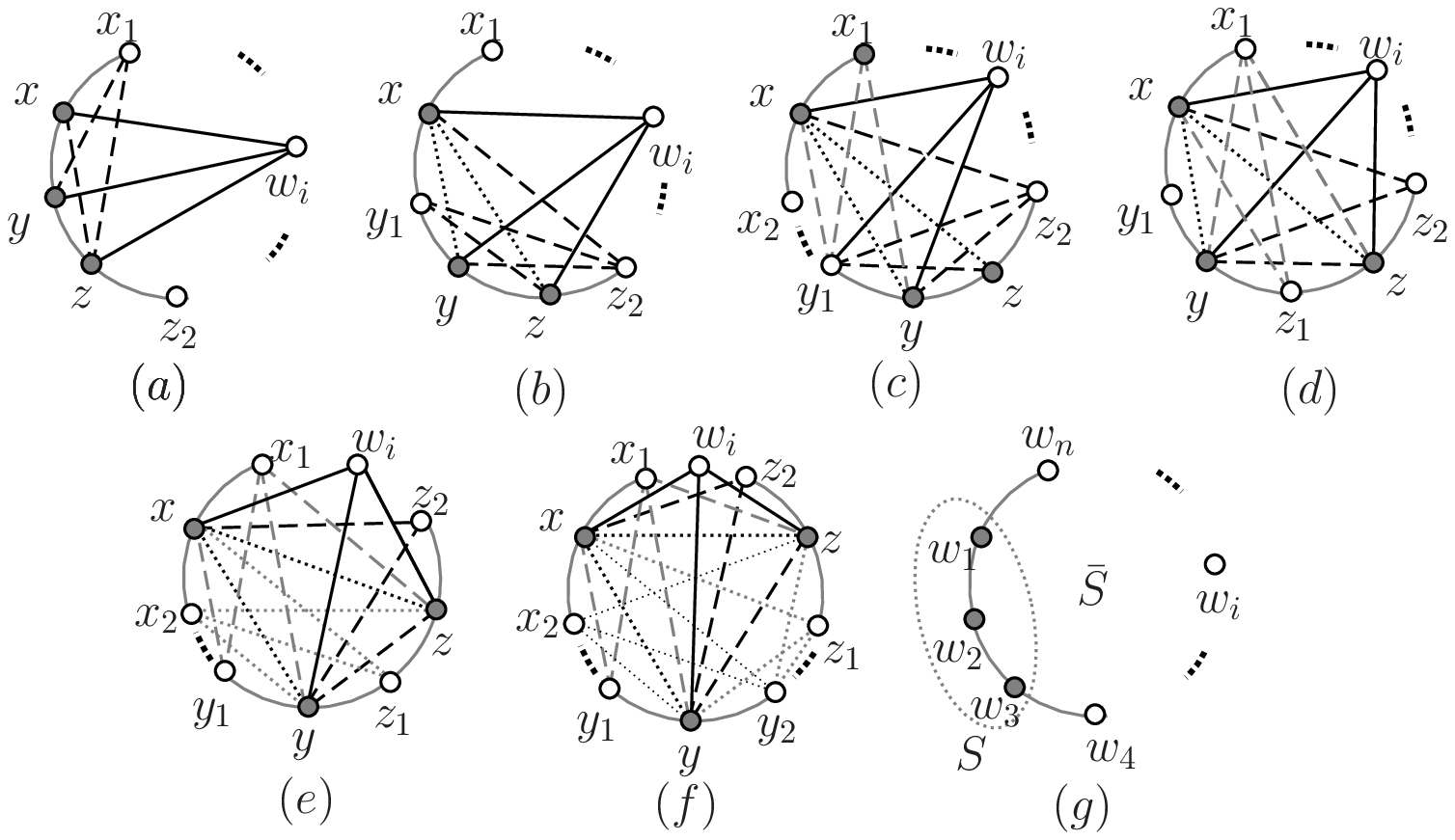}}\\
Figure 5. Graphs for Example $3$.
\end{center}
\end{figure}

If $d_{C_n}(x,y)\geq 3$ and $d_{C_n}(y,z)=1$, without loss of
generality, let $N_{C_n}(x)=\{x_1,x_2\}$ and $N_{C_n}(z)=\{y_1,z\}$
and $N_{C_n}(z)=\{y,z_2\}$, then the trees $T_i= xw_i\cup yw_i\cup
zw_i$ together with $T_1= xy\cup xz$ and $T_2= z_2x\cup z_2y\cup
z_2y_1\cup y_1z$ and $T_3= xy_1\cup y_1x_1\cup x_1y\cup x_1z$ form
$n-4$ edge-disjoint $S$-trees (See Figure 5 $(c)$), namely,
$\lambda(S)\geq n-4$, where
$\{w_1,w_2,\cdots,w_{n-7}\}=V(G)\setminus
\{x,y,z,x_1,x_2,y_1,z_2\}$.

If $d_{C_n}(x,y)=2$ and $d_{C_n}(y,z)=2$, without loss of
generality, let $N_{C_n}(x)=\{x_1,y_1\}$ and
$N_{C_n}(z)=\{y_1,z_1\}$ and $N_{C_n}(z)=\{z_1,z_2\}$, then the
trees $T_i= xw_i\cup yw_i\cup zw_i$ together with $T_1= xz\cup xy$
and $T_2= xz_2\cup yz_2\cup yz$ and $T_3= x_1y\cup x_1z\cup
x_1z_1\cup xz_1$ form $n-4$ edge-disjoint $S$-trees (See Figure 5
$(d)$), namely, $\lambda(S)\geq n-4$, where
$\{w_1,w_2,\cdots,w_{n-7}\}=V(G)\setminus
\{x,y,z,x_1,y_1,z_1,z_2\}$.

If $d_{C_n}(x,y)\geq 3$ and $d_{C_n}(y,z)=2$, without loss of
generality, let $N_{C_n}(x)=\{x_1,x_2\}$ and
$N_{C_n}(z)=\{y_1,z_1\}$ and $N_{C_n}(z)=\{z_1,z_2\}$, then the
trees $T_i= xw_i\cup yw_i\cup zw_i$ together with $T_1= xz\cup xy$
and $T_2= xz_2\cup z_2y\cup yz$ and $T_3= x_1y\cup x_1z\cup
x_1y_1\cup xy_1$ and $T_4= x_2y\cup x_2z\cup x_2z_1\cup z_1x$ form
$n-4$ edge-disjoint $S$-trees (See Figure 5 $(e)$), namely,
$\lambda(S)\geq n-4$, where
$\{w_1,w_2,\cdots,w_{n-8}\}=V(G)\setminus
\{x,y,z,x_1,x_2,y_1,y_2,z_2\}$.

Suppose that $d_{C_n}(x,y)\geq 3$ and $d_{C_n}(y,z)\geq 3$, without
loss of generality, let $N_{C_n}(x)=\{x_1,x_2\}$ and
$N_{C_n}(z)=\{y_1,y_2\}$ and $N_{C_n}(z)=\{z_1,z_2\}$. Then the
trees $T_i= xw_i\cup yw_i\cup zw_i$ together with $T_1= xz\cup xy$
and $T_2= xz_2\cup yz_2\cup yz$ and $T_3= xz_1\cup yz_1\cup
y_2z_1\cup y_2z$ and $T_4= x_1y\cup x_1z\cup x_1y_1\cup y_1x$  and
$T_5= x_2y\cup x_2z\cup x_2y_2\cup y_2x$ form $n-4$ edge-disjoint
$S$-trees (See Figure 5 $(f)$), namely, $\lambda(S)\geq n-4$, where
$\{w_1,w_2,\cdots,w_{n-9}\}=V(G)\setminus
\{x,y,z,x_1,x_2,y_1,y_2,z_1,z_2\}$.

From the arbitrariness of $S$, we know that
$\lambda_3(\overline{G})\geq n-4$ by definition. Now we show that
$\lambda_3(\overline{G})\leq n-4$ for $\overline{G}=\overline{C_n}$.
Choose $S=\{w_1,w_2,w_3\}\subseteq V(G)=V(C_n)$. Then $w_1w_n\in
E(C_n)$ and $w_3w_4\in E(C_n)$. Thus $|E(\overline{G}[S])|=1$ and
$|E_{\overline{G}}[S,\bar{S}]|=3(n-3)-2$, which implies that
$|E(\overline{G}[S])\cup E_{\overline{G}}[S,\bar{S}]|=3(n-3)-1$ (See
Figure 5 $(g)$). One can see that each tree connecting $S$ in
$\overline{G}$ uses at least $3$ edges from $E(\overline{G}[S])\cup
E_{\overline{G}}[S,\bar{S}]$. Therefore $\lambda_3(\overline{G})\leq
\frac{3(n-3)-1}{3}=n-3-\frac{1}{3}$, which results in
$\lambda_3(\overline{G})\leq n-4$ since $\lambda_3(\overline{G})$ is
an integer. So $\lambda_3(\overline{G})=n-4$ and
$\lambda_3(G)\cdot\lambda_3(\overline{G})=\lambda_3(C_n)\cdot\lambda_3(\overline{C_n})
=1\cdot(n-4)=(\frac{2m}{n}-1)(n-2-\frac{2m}{n})$. The upper bound is
sharp.

For $m\geq n-1$ and $\frac{2m}{n}=r+\ell(1\leq \ell \leq n-1)$, let
$G=P_4$. Then
$\lambda_3(G)=1=\lfloor\frac{6}{4}\rfloor=\lfloor\frac{2m}{n}\rfloor$
and
$\lambda_3(\overline{G})=\lambda_3(P_4)=1=4-2-\lfloor\frac{6}{4}\rfloor
=n-2-\lfloor\frac{2m}{n}\rfloor$. So
$\lambda_3(G)\cdot\lambda_3(\overline{G})
=\lfloor\frac{2m}{n}\rfloor(n-2-\lfloor\frac{2m}{n}\rfloor)$.

Combining with Lemmas \ref{lem8} and \ref{lem9}, we can obtain the
following result.

\begin{theorem}\label{th4}
Let $G\in \mathcal {G}(n,m)$. For $n\geq 6$, we have

$(1)$ $L(n,m)\leq \lambda_k(G)+\lambda_k(\overline{G})\leq M(n,m)$;

$(2)$ $0\leq \lambda_k(G)\cdot \lambda_k(\overline{G})\leq N(n,m)$,

where $L(n,m), M(n,m), N(n,m)$ are defined in Lemmas \ref{lem8} and
\ref{lem9}.

Moreover, the upper and lower bounds are sharp.
\end{theorem}


\small
\begin{thebibliography}{11}

\bibitem{Achuthan} N. Achuthan, N. R. Achuthan, L. Caccetta,
\emph{On the Nordhaus-Gaddum problems}, Australasian J. Combin.
2(1990), 5-27.

\bibitem{Alavi} Y. Alavi, J. Mitchem, \emph{The connectivity
and edge-connectivity of complementary graphs}, Lecture Notes in
Math. 186(1971), 1-3.

\bibitem{Aouchiche} M. Aouchiche, P. Hansen, \emph{A survey of
Nordhaus-Gaddum type relations}, Discrete Appl. Math., 2012.

\bibitem{bondy} J. Bondy, U. Murty, \emph{Graph Theory},
GTM 244, Springer, 2008.

\bibitem{Chartrand1} G. Chartrand, S.F. Kappor, L. Lesniak, D.R. Lick,
\emph{Generalized connectivity in graphs}, Bull. Bombay Math.
Colloq. 2(1984), 1-6.

\bibitem{Chartrand2} G. Chartrand, F. Okamoto, P. Zhang,
\emph{Rainbow trees in graphs and generalized connectivity},
Networks 55(4)(2010), 360-367.

\bibitem{Grotschel1} M. Gr\"{o}tschel,
\emph{The Steiner tree packing problem in $VLSI$ design}, Math.
Program. 78(1997), 265-281.

\bibitem{Grotschel2} M. Gr\"{o}tschel, A. Martin, R. Weismantel,
\emph{Packing Steiner trees: A cutting plane algorithm and
commputational results}, Math. Program. 72(1996), 125-145.

\bibitem{Harary} F. Harary, \emph{The maximum connectivity of a graph},
Proc. Nat. Acad. Sci. USA 1142-1146.

\bibitem{LLSun} H. Li, X. Li, Y. Sun, \emph{The generalied $3$-connectivity
of Cartesian product graphs}, Discrete Math. Theor. Comput. Sci.
14(1)(2012), 43-54.

\bibitem{LL} S. Li, X. Li, Note on the hardness of generalized
connectivity, J. Combin. Optim. 24(2012), 389-396.

\bibitem{LLZ} S. Li, X. Li, W. Zhou, \emph{Sharp bounds for the
generalized connectivity $\kappa_3(G)$}, Discrete Math. 310(2010),
2147-2163.

\bibitem{LMS} X. Li, Y. Mao, Y. Sun, \emph{On the generalized
(edge-)connectivity}, arXiv:1112.0127 [math.CO] 2011.

\bibitem{Nash} C.St.J.A, Nash-Williams,
\emph{Edge-disjonint spanning trees of finite graphs}, J. London
Math. Soc. 36(1961), 445-450.

\bibitem{Sherwani} N. Sherwani, \emph{Algorithms for VLSI physical
design automation}, 3rd Edition, Kluwer Acad. Pub., London, 1999.

\bibitem{Tutte} W. Tutte,
\emph{On the problem of decomposing a graph into $n$ connected
factors}, J. London Math. Soc. 36(1961), 221-230.

\bibitem{Palmer} E. Palmer, \emph{On the spanning tree packing number
of a graph: a survey}, Discrete Math. 230(2001), 13-21.

\bibitem{Kriesell1} M. Kriesell, \emph{Edge-disjoint trees containing
some given vertices in a graph}, J. Combin. Theory, Ser.B, 88(2003),
53-65.

\bibitem{Kriesell2} M. Kriesell, \emph{Edge-disjoint Steiner trees
in graphs without large bridges}, J. Combin. Theory, Ser.B,
62(2009), 188-198.

\bibitem{West} H. Wu, D. West, \emph{Packing Steiner trees and
$S$-connectors in graphs}, J. Combin. Theory, Ser.B, 102(2012),
186-205.

\end{thebibliography}
\end{document}